\documentclass[12pt]{amsart}

\usepackage[T1]{fontenc}
\usepackage{mathpazo}
\usepackage[none]{hyphenat}
\usepackage{mathpple}
\usepackage{amsmath,amsfonts}
\usepackage{amsthm}
\usepackage[all]{xy}
\usepackage{color}
\usepackage{tikz}

\usepackage{hyperref}

\numberwithin{equation}{section}

\newcommand{\on}{\operatorname}
\newcommand{\mf}{\mathfrak}
\newcommand{\mc}{\mathcal}

\newcommand{\fun}{\mathbb{F}_1}

\newcommand{\mspec}{\on{MSpec}}
\newcommand{\mproj}{\on{MProj}}
\newcommand{\p}{\mathfrak{p}}

\newcommand{\nc}{\newcommand}

\nc{\h}{\mathfrak{h}}
\nc{\g}{\mathfrak{g}}
\nc{\n}{\mathfrak{n}}
\nc{\ch}{\on{CH}}
\nc{\wt}{\widetilde}

\nc{\F}{\mc{F}}
\nc{\C}{\mc{C}}
\nc{\M}{\on{M}}
\nc{\T}{\mc{T}}

\nc{\G}{\mc{G}}
\nc{\ov}{\overline}

\nc{\VFun}{\on{Vect}(\fun)}

\nc{\FF}{\mathbb{F}}

\nc{\Amod}{$A$-{mod}}

\nc{\Pone}{\mathbb{P}^1}
\nc{\Aone}{\mathbb{A}^1}
\renewcommand{\fun}{\mathbb{F}_1}
\renewcommand{\mf}{\mathfrak}
\nc{\slthat}{\widehat{\mf{sl}}_2}
\renewcommand{\a}{\mathfrak{a}}
\renewcommand{\p}{\mathfrak{p}}
\nc{\spec}{\on{MSpec}}
\nc{\Msch}{\mc{M}sch}
\nc{\mt}{\widetilde{M}}

\theoremstyle{plain}
\newtheorem{theorem}{Theorem}[section]
\newtheorem{thm-defn}{Theorem/Definition}[section]
\newtheorem{lemma}{Lemma}[section]
\newtheorem{lem-defn}{Lemma/Definition}[section]

\newtheorem{corollary}{Corollary}[section]

\theoremstyle{definition}
\newtheorem{rmk}{Remark}[section]

\newtheorem{example}{Example}
\newtheorem{definition}{Definition}
\newtheorem*{theorem*}{Theorem}
\newtheorem*{corollary*}{Corollary}

\theoremstyle{remark}

\numberwithin{example}{section}
\numberwithin{definition}{section}

\allowdisplaybreaks[4]  %%%%%%%%%%% choose 1-4, where 4 is the strongest desire to breack

\begin{document}

 \title{Quasicoherent sheaves on projective schemes over $\fun$}
  \author{ Oliver Lorscheid { and } Matt Szczesny}
  \date{}

  \maketitle

\begin{abstract}
Given a graded monoid $A$ with $1$, one can construct a projective monoid scheme $\mproj(A)$ analogous to $\on{Proj}(R)$ of a graded ring $R$. This paper is concerned with the study of quasicoherent sheaves on $\mproj(A)$, and we prove several basic results regarding these. We show that: 
\begin{enumerate}
\item every quasicoherent sheaf $\F$ on $\mproj(A)$ can be constructed from a graded $A$--set in analogy with the construction of quasicoherent sheaves on $\on{Proj}(R)$ from graded $R$--modules
\item if $\F$ is coherent on $\mproj(A)$, then $\F(n)$ is globally generated for large enough $n$, and consequently, that $\F$ is a quotient of a finite direct sum of invertible sheaves
\item if $\F$ is coherent on $\mproj(A)$, then $\Gamma(\mproj(A), \F)$ is a finite pointed set
\end{enumerate}  
The last part of the paper is devoted to classifying coherent sheaves on $\mathbb{P}^1$ in terms of certain directed graphs and gluing data. The classification of these over $\fun$ is shown to be much richer and combinatorially interesting than in the case of ordinary $\mathbb{P}^1$, and several new phenomena emerge. 
\end{abstract}

%%%%%%%%%%%%%%%%%%%%%%%%%%%%%%
%% ǰ Բ   
%%%%%%%%%%%%%%%%%%%%%%%%%%%%%%

%\begin{abstract} We explain the homological relation between the Frobenius structure on the  deformation space of Calabi-Yau manifold and the gauge theory of  Kodaira-Spencer gravity. We show that the genus zero generating function of descendant invariants on Calabi-Yau manifolds from Barannikov's semi-infinite variation of Hodge structures is equivalent to the Kodaira-Spencer gauge theory at tree level.
%\end{abstract}

\section{Introduction}

The last twenty years have seen the development of several notions of "algebraic geometry over $\fun$". This quest is motivated by a variety of questions in arithmetic, representation theory, algebraic geometry, and combinatorics. Since several surveys \cite{LPL2, Lor3, Th2} of the motivating ideas exist, we will not attempt to sketch them here. One of the simplest approaches to algebraic geometry over $\fun$ is via the theory of monoid schemes, originally developed by Kato \cite{Kato}, Deitmar \cite{D}, and Connes-Consani \cite{CC2}. Here, the idea is to replace prime spectra of commutative rings which are the local building blocks of ordinary schemes, by prime spectra of commutative unital monoids with $0$. One then obtains a topological space $X$ with a structure sheaf of commutative monoids $\mc{O}_X$. In this setting, one can define a notion of \emph{quasicoherent sheaf} on $(X, \mc{O}_X)$, as a sheaf of pointed sets carrying an action of $\mc{O}_X$, which for each affine $U \subset X$ is described by an $\mc{O}_X (U)$--set. Imposing a condition of local finite generation yields a notion of \emph{coherent sheaf}. 

This paper is devoted to the study of quasicoherent and coherent sheaves on projective monoid schemes. Given a graded commutative unital monoid with $0$, $A = \bigoplus_{n \in \mathbb{N}} A_n$, one can form a monoid scheme $\mproj(A)$ in a manner analogous to the $\on{Proj}$ construction in the setting of graded rings. We call such a monoid scheme \emph{projective}. In analogy with the setting of ordinary schemes, we construct a functor:
\begin{align*}
grA-Mod & \rightarrow Qcoh(X)\\
M & \rightarrow \widetilde{M}\\
\end{align*}
from the category of graded (set-theoretic) $A$--modules to the category of quasicoherent sheaves on $\mproj(A)$. It sends $M$ to the quasicoherent sheaf  $\wt{M}$ whose sections over the affine open $\mproj(A_f)$ are $M_{(f)}$ - the degree zero elements of the localization of $M$ with respect to the multiplicative subset generated by $f$.  We prove that every quasicoherent sheaf on $\mproj (A)$ arises via this construction, and that as in the case of ordinary schemes, there is a canonical representative:

\begin{theorem*}[\ref{mainthm1}]
Let $A$ be a graded monoid finitely generated by $A_1$ over $A_0$, and let $X = \mproj(A)$. Given a quasicoherent sheaf $\mc{F}$ on $X$, there exists a natural isomorphism $\beta: \wt{\Gamma_* ( \mc{F})} \simeq \mc{F}$, where $ \Gamma_* ( \mc{F} ) = \oplus_{n \in \mathbb{Z}} \Gamma(\mproj(A), \F(n)) $.
\end{theorem*}

This result allows a purely combinatorial classification of quasicoherent sheaves on $\mproj(A)$ in terms of graded $A$--sets. We use this to obtain $\fun$--analogs of other key  foundational results on quasicoherent  and coherent sheaves on projective schemes, such as the following regarding global finite generation:

\begin{theorem*}[\ref{fg_sheaf}]
Let $A$ be a graded monoid finitely generated by $A_1$ over $A_0$, and $\F$ a coherent sheaf on $X=\mproj(A)$. Then there exists $n_0$ such that $\F(n)$ is generated by finitely many global sections for all $n \geq n_0$. 
\end{theorem*}

As a corollary, we obtain:

\begin{corollary*}[\ref{maincor}]
With the hypotheses of Theorem \ref{fg_sheaf}, there exist integers $m \in \mathbb{Z}, k \geq 0$, such that $\F$ is a quotient of $\mc{O}_X(m)^{\oplus k}$
\end{corollary*}

One of the key properties of a coherent sheaf on a projective scheme over a field $k$ is the finite-dimensionality of the space of global sections. We obtain the following $\fun$--analog:

\begin{theorem*}[\ref{mainthm3}]
Let $A$ be a graded monoid finitely generated by $A_1$ over $A_0=\langle 0,1 \rangle$, and $\F$ a coherent sheaf on $X=\mproj(A)$. Then $\Gamma(X,\F)$ is a finite pointed set. 
\end{theorem*}

The last section is devoted to the study and classification of coherent sheaves on the simplest non-trivial projective scheme - $\mathbb{P}^1$. By viewing $\mathbb{P}^1$ as two copies of $\mathbb{A}^1$ glued together as $\mathbb{A}^1_{0} \cup \mathbb{A}^1_{\infty}$, we give a combinatorial description of the indecomposable coherent sheaves in terms of certain directed graphs and gluing data along the intersection $\mathbb{A}^1_{0} \cap \mathbb{A}^1_{\infty}$. The classification is much richer than in the case of $\mathbb{P}^1_k$ for $k$ a field, as base change to $\on{Spec}(k)$ identifies many  coherent sheaves non-isomorphic over $\fun$. Example \ref{secondexample} exhibits an unusual phenomenon possible over $\fun$ but impossible over $\on{Spec}(k)$: a pair of coherent sheaves $\F, \F'$ with infinitely many non-isomorphic extensions of $\F$ by $\F'$ (though these yield a finite-dimensional space of extensions upon based-change to $\on{Spec}(k)$). 

\bigskip

\noindent{\bf Acknowledgements:} M.S. gratefully acknowledges the support of a Simons Foundation Collaboration Grant during the writing of this paper.

\section{Monoid schemes} \label{msch}

In this section, we briefly recall the notion of a monoid scheme following \cite{CHWW, D}, which we will use as our model for algebraic geometry over $\fun$. This is essentially equivalent to the notion of $\mathfrak{M}_0$-scheme in the sense of \cite{CC2}  For other (some much more general) approaches to schemes over $\fun$, see \cite{H2,Du,LPL2, Lor3, Sou,TV}. Recall that ordinary schemes are ringed spaces locally modeled on affine schemes, which are spectra of commutative rings.  A monoid scheme is locally modeled on an affine monoid scheme, which is the spectrum of a commutative unital  monoid with $0$. In the following, we will denote monoid multiplication by juxtaposition or "$\cdot$". In greater detail:

A \emph{monoid} $A$ will be a commutative associative monoid with identity $1_A$ and zero $0_A$ (i.e. the absorbing element). We require
\[
 1_A \cdot a = a \cdot 1_A = a \hspace{1cm} 0_A \cdot a = a \cdot 0_A = 0_A \hspace{1cm} \forall a \in A
\]
Maps of monoids are required to respect the multiplication as well as the special elements $1_A, 0_A$. An \emph{ideal} of $A$ is a nonempty subset  $\a \subset A$ such that $A \cdot \a \subset \a$. An ideal $\p \subset A$ is \emph{prime} if $xy \in \p$ implies either $x \in \p$ or $y \in \p$. 

Given a monoid $A$, $\spec (A)$ is defined to be the topological space with underlying set $$\spec(A) := \{ \p | \p \subset A \textrm{ is a prime ideal } \}, $$ equipped with the Zariski topology, whose closed sets are of the form $$ V(\a) := \{ \p | \a \subset \p, \p \textrm{ prime } \},  $$ where $\a$ ranges over all ideals of $A$. 
Given a multiplicatively closed subset $S \subset A$, the \emph{localization of $A$ by $S$}, denoted $S^{-1}A$, is defined to be the monoid consisting of symbols $\{ \frac{a}{s} |  a \in A, s \in S \}$, with the equivalence relation $$\frac{a}{s} = \frac{a'}{s'}  \iff \exists \; s'' \in S \textrm{ such that } as's'' = a's s'', $$ and multiplication is given by $\frac{a}{s} \times \frac{a'}{s'} = \frac{aa'}{ss'} $. 

For $f \in A$, let $S_f$ denote the multiplicatively closed subset $\{ 1, f, f^2, f^3, \cdots,  \}$. We denote by $A_f$ the localization $S^{-1}_f A$, and by $D(f)$ the open set $\spec(A) \backslash V(f) \simeq \spec A_f$, where $V(f) := \{ \p \in \spec (A) | f \in \p \}$. The open sets $D(f)$ cover $\spec (A)$. 
$\spec (A)$ is equipped with a  \emph{structure sheaf} of monoids $\mc{O}_{\spec(A)}$, satisfying the property $\Gamma(D(f), \mc{O}_{\spec(A)}) = A_f$.    Its stalk at $\p \in \spec A$ is $A_{\p} := S^{-1}_{\p} A$, where $S_{\p} = A \backslash \p$. 

A unital homomorphism of monoids $\phi: A \rightarrow B$ is \emph{local} if $\phi^{-1}(B^{\times}) \subset A^{\times}$, where $A^{\times}$ (resp. $B^{\times}$) denotes the invertible elements in $A$ (resp. $B$). 
A \emph{monoid space} is a pair $(X, \mc{O}_X)$ where $X$ is a topological space and $\mc{O}_X$ is a sheaf of monoids. A \emph{morphism of monoid spaces} is a pair $(f, f^{\#})$ where $f: X \rightarrow Y$ is a continuous map, and $f^{\#}: \mc{O}_Y \rightarrow f_* \mc{O}_X$ is a morphism of sheaves of monoids, such that the induced morphism on stalks $f^{\#}_\p : \mc{O}_{Y, f(\p)} \rightarrow f_* \mc{O}_{X, \p}$ is local. 
An \emph{affine monoid scheme} is a monoid space isomorphic to $(\spec (A), \mc{O}_{\spec(A)})$. 
Thus, the category of affine monoid schemes is opposite to the category of monoids. 
A monoid space $(X,\mc{O}_X)$ is called a \emph{monoid scheme}, if for every point $x \in X$ there is an open neighborhood $U_x \subset X$ containing $x$ such that $(U_x, \mc{O}_X \vert_{U_x})$ is an affine monoid scheme. We denote by $\Msch$ the category of monoid schemes.

\begin{example} \label{P1}

Denote by $\langle t \rangle$ the free commutative unital monoid with zero generated by $t$, i.e.
\[
\langle t \rangle := \{ 0, 1, t, t^2, t^3, \cdots, t^n, \cdots \},
\]
and let $\mathbb{A}^1 := \on{MSpec} ( \langle t \rangle )$ - the monoid affine line.
Let $\langle t, t^{-1} \rangle$ denote the monoid 
\[
\langle t,t^{-1} \rangle := \{ \cdots, t^{-2}, t^{-1}, 1, 0, t, t^2, t^3, \cdots \}.
\]
We obtain the following diagram of inclusions
\begin{equation*}
\langle t \rangle \hookrightarrow \langle t,t^{-1} \rangle \hookleftarrow \langle t^{-1} \rangle.
\end{equation*}
Taking spectra, and denoting by $U_0 = \on{MSpec} ( \langle t \rangle ) , U_{\infty} = \on{MSpec} ( \langle t^{-1} \rangle) $, we obtain
\begin{equation*}
\mathbb{A}^1 \simeq U_0 \hookleftarrow U_0 \cap U_{\infty} \hookrightarrow U_{\infty} \simeq \mathbb{A}^{1}
\end{equation*}
We define $\mathbb{P}^{1}$, the monoid projective line, to be the monoid scheme obtained by gluing two copies of $\mathbb{A}^1$ according to this diagram. It has three points - two closed points $0 \in U_0, \;  \infty \in U_{\infty}$, and the generic point $\eta$. Denote by $\iota_{0} : U_0 \hookrightarrow \Pone$, $\iota_{\infty}: U_{\infty} \hookrightarrow \Pone$  the corresponding inclusions. 

\end{example}

\subsection{Base Change}

Given a commutative ring $R$, there exists a base-change functor
\begin{align*}
 \Msch & \rightarrow \on{Sch} / \on{Spec} R \\
 X & \rightarrow X_R
\end{align*}
It is defined on affine schemes by
\[
(\mspec{A})_{R} = \on{Spec} R[A] 
\]
where $R[A]$ is the monoid algebra:
\[
R[A] := \left\{ \sum r_i a_i | a_i \in A, a_i \neq 0, r_i \in R \right\}
\]
with multiplication induced from the monoid multiplication. 
For a general monoid scheme $X$, $X_R$ is defined by gluing the open affine subfunctors of $X$.

\section{$A$-modules} \label{Amod}
%\subsubsection{Vector spaces over $\fun$}

%In this section we recall the category of vector spaces over $\fun$ following \cite{KapS, H}. 
%\begin{definition}
%The category $\VFun$ of vector spaces over $\fun$ is defined as follows. 
%\begin{eqnarray*}
%\on{Ob}(\VFun) & := & \{\textrm{ pointed sets } (V, *_V ) \} \\
%\on{Hom}(V,W) & := & \{ \textrm{ maps } f: V \rightarrow W \vert \; f(*_V) = *_W \\ & & f \vert_{V \backslash f^{-1}(*_W)} \textrm{ is an injection }\}
%\end{eqnarray*}
%Composition of morphisms is defined as the composition of maps of sets, and so is associative. We refer to the unique morphism $f \in \on{Hom}(V,W)$ such that $f(V)= *_W$ as the zero map. If $V \in \VFun$ is a finite set, we define the \emph{dimension} of $V \in \VFun$ as $\dim(V) := |V| := \#V -1$ (i.e. we do not count the basepoint). 
%\end{definition}

%$\on{Hom}(V,W)$ is a pointed set, with distinguished element the zero map. Thus \\ \mbox{ $\on{Hom}(V,W) \in \VFun$.} For a fixed $V \in \VFun$, $\on{End}(V) := \on{Hom}(V,V) $ has the structure of a (generally non-commutative) semi-group with $0$ (the zero map) and $1$ (the identity map). 

%\subsubsection{$A$--modules} \label{Amod}
In this section, we briefly recall the properties of set-theoretic $A$-modules following \cite{CLS}. 

Let $A$ be a monoid. An \emph{$A$--module} is a pointed set $(M,*_{M})$ together with an action 
\begin{align*}
\mu: A \times M & \rightarrow M \\
(a,m) & \rightarrow a\cdot m
\end{align*}
which is compatible with the monoid multiplication (i.e. $1_A \cdot m = m$, $a \cdot (b \cdot m) = (a \cdot b) \cdot m$), and $0_A \cdot m = *_{M} \; \forall m \in M$). 
We will refer to elements of $M \backslash *_M$ as \emph{nonzero} elements. 

A \emph{morphism of $A$--modules} $f: (M,*_M) \rightarrow (N,*_N)$ is a map of pointed sets (i.e. we require $f(*_M) = *_M$) compatible with the action of $A$, i.e. $f(a \cdot m) = a \cdot f(m)$.

A pointed subset $(M',*_M) \subset (M,*_M)$ is called an \emph{$A$--sub-module} if $A \cdot M' \subset M'$. In this case we may form the quotient module $M/M'$, where $M/M' := M \backslash (M' \backslash *_M)$, $*_{M/M'} = *_M$, and the action of $A$ is defined by setting 
$$a \cdot \overline{m} = \left\{ \begin{array}{ll} \overline{a \cdot m} & \textrm{if } a \cdot m \notin M' \\ *_{M/M'} & \textrm{ if } a\cdot m \in M'    \end{array} \right. $$  where $\overline{m}$ denotes $m$ viewed as an element of $M/M'$. If $M$ is finite, we define $|M| = \# M -1 $, i.e. the number of non-zero elements. 

\bigskip

Denote by $\Amod$ the category of $A$--modules. It has the following properties:
\begin{enumerate}
\item $\Amod$ has a zero object $\emptyset= \{ * \}$ - the one-element pointed set. 
\item A morphism $f: (M,*_M) \rightarrow (N,*_N)$ has a kernel $(f^{-1}(*_N), *_M)$ and a cokernel $N/\on{Im}(f)$.
\item $\Amod$ has a symmetric monoidal structure $$M \oplus N := M \vee N := M \sqcup N / *_M \sim *_N$$ which we will call "direct sum", with identity object $\{*\}$. 
\item If $R \subset M \oplus N$ is an $A$--submodule, then $R = (R \cap M) \oplus (R \cap N) $.  
\item $\Amod$ has a symmetric monoidal structure  $$M \otimes_{A} N :=  M \times N / \sim, $$ where $\sim$ is the equivalence relation generated by  $$(a \cdot m, n) \sim (m, a \cdot n), \; a \in A,$$ with identity object $\{ A \}$. 
\item $\oplus, \otimes$ satisfy the usual associativity and distributivity properties. 
\end{enumerate}

\bigskip

$M \in \; \Amod$ is \emph{finitely generated} if there exists a surjection $\oplus^n_{i=1} A \rightarrow M$ of $A$--modules for some $n$. Explicitly, this means that there are $m_1, \cdots, m_n \in M$ such that for every $m \in M$, $m = a \cdot m_i$ for some $1\leq i \leq n$, and we refer to the $m_i$ as \emph{generators}. $M$ is said to be \emph{free of rank n} if $M \simeq \oplus^n_{i=1} A$. 
For an element $m \in M$, define $$Ann_A (m) := \{ a \in A | a \cdot m = *_M \}.$$  Obviously $0_A \subset Ann_A (m) \;  \forall m \in M$. An element $m \in M$ is \emph{torsion} if $Ann_A (m) \neq 0_A$. The subset of all torsion elements in $M$ forms an $A$--submodule, called the \emph{torsion submodule} of $M$, and denoted $M_{tor}$. An $A$--module is \emph{torsion-free} if $M_{tor} = \{ *_M \}$ and \emph{torsion} if $M_{tor} = M$. 
%Every $M$ can be uniquely written $M = M_{tor} \oplus M_{tf}$, where $M_{tf}$ is torsion-free. 
We define the \emph{length} of a torsion module $M$ to be $|M|$.

\section{Quasicoherent sheaves} \label{coh}

In this section, we briefly recall the definitions and properties of quasicoherent and coherent sheaves over monoid schemes. We refer the reader to \cite{CLS} for details. 

Given a multiplicatively closed subset $S \subset A$ and an $A$--module $M$, we may form the $S^{-1} A$--module $S^{-1}M$, where
\[
S^{-1} M := \{ \frac{m}{s} \vert \; m \in M, s \in S \} 
\]
with the equivalence relation 
$$\frac{m}{s} = \frac{m'}{s'}  \iff \exists \; s'' \in S \textrm{ such that } s's'' m = s s'' m', $$
where the $S^{-1}A$--module structure is given by $\frac{a}{s} \cdot \frac{m}{s'} := \frac{am}{ss'}.$ For $f \in A$, we define $M_f$ to be $S^{-1}_f M$. 

Let $X$ be a topological space, and $\mc{A}$ a sheaf of monoids on $X$. We say that a sheaf of pointed sets $\mc{M}$ is an \emph{$\mc{A}$--module} if for every open set $U \subset X$, $\mc{M}(U)$ has the structure of an $\mc{A}(U)$--module with the usual compatibilities. In particular, given a monoid $A$ and an $A$--module $M$, there is a sheaf of $\mc{O}_{\spec(A)}$--modules $\wt{M}$ on $\on{MSpec} (A)$, defined on basic affine sets $D(f)$ by $\wt{M}(D(f)) := M_f $.  
For a monoid scheme $X$, a sheaf of $\mc{O}_X$--modules $\F$ is said to be \emph{quasicoherent} if for every $x \in X$  there exists an open affine $U_x \subset X$ containing $x$ and an $O_X (U_x)$--module $M$ such that $\F \vert_{U_x} \simeq \wt{M}$.  $\F$ is said to be \emph{coherent}
if $M$ can always be taken to be finitely generated, and \emph{locally free} if $M$ can be taken to be free. Please note that here our conventions are different from \cite{D}. If $X$ is connected, we can define the \emph{rank} of a locally free sheaf $\F$ to be the rank of the stalk $\F_{x}$ as an $\mc{O}_{X,x}$--module for any $x \in X$. A locally free sheaf of rank one will be called \emph{invertible}. 

\begin{rmk} \label{Noeth_rmk}
The notion of coherent sheaf on ordinary schemes is well-behaved only for schemes that are locally Noetherian. The corresponding notion for monoid schemes is being of \emph{finite type}, and is introduced at the end of section \ref{projschemes}. We will consider coherent sheaves only on monoid schemes satisfying this property. 
\end{rmk}

For a monoid $A$, there is an equivalence of categories between the category of quasicoherent sheaves on $\on{Spec} A$ and the category of $A$--modules, given by $\Gamma(\on{Spec} A, \cdot)$. A quasicoherent sheaf $\F$ on $X$ is \emph{torsion} (resp. \emph{torsion-free}) if $\F(U)$ is a torsion $\mc{O}_X (U)$--module (resp. torsion-free $\mc{O}_X (U)$--module ) for every open affine $U \subset X$. 
 
The operations $\oplus, \otimes$ induce analogous ones on $\mc{O}_X$--modules. More precisely, for $\mc{O}_X$--modules $\F, \G$ and an open subset $U \subset X$ we define $\F \oplus \G (U):= \F(U) \oplus G(U)$ with the obvious $\mc{O}_X (U)$ structure. $\F \otimes_{O_X} \G$ is defined to be the sheaf associated to the presheaf $U \rightarrow \F(U) \otimes_{\mc{O}_X (U)} \G(U)$. If $X = \spec(A)$ is affine, and $M, N$ are $A$--modules,  we have $\wt{M \oplus N} = \wt{M} \oplus \wt{N}$ and \\ $\wt{M \otimes_A N} = \wt{M} \otimes_{\mc{O}_{\spec (A)}} \wt{N}$. This implies that on an arbitrary monoid scheme $X$, quasicoherent and coherent sheaves are closed under $\oplus$ and $\otimes$.  

%\medskip
%
%\begin{rmk} \label{ttf}
%It is clear that if $\F$ is torsion and $\F'$ torsion-free, then $\on{Hom}(\F, \F') = 0$. 
%\end{rmk}

\medskip

\begin{rmk} \label{subobjects}
It follows from property $(4)$ of the category $\Amod$ that if $\F, \F'$ are quasicoherent $\mc{O}_X$--modules, and $\mc{G} \subset \F \oplus \F'$ is an $\mc{O_X}$--submodule, then \mbox{$\mc{G} = (\mc{G}\cap \F) \oplus (\mc{G} \cap \F' )$}, where for an open subset $U \subset X$,  \begin{equation} (\mc{G} \cap \F) (U) := \mc{G}(U) \cap \F (U). \end{equation}
\end{rmk}

\bigskip

If $X$ is a monoid scheme, we will denote by  $QCoh(X)$ (resp. $Coh(X)$) the category of quasicoherent  (resp. coherent) \mbox{$\mc{O}_X$--modules} on $X$. It follows from the properties of the category $\Amod$ listed in section \ref{Amod} that $QCoh(X)$ possesses a zero object $\emptyset$ (defined as the zero module $\emptyset$ on each open affine $\spec A \subset X$), kernels and co-kernels, as well as monoidal structures $\oplus$ and $\otimes$. We may therefore talk about exact sequence in $QCoh(X)$. A short exact sequence isomorphic to one of the form
\[
\emptyset \rightarrow \F \rightarrow \F \oplus \mc{G} \rightarrow \mc{G} \rightarrow \emptyset
\]
is called \emph{split}.  A coherent sheaf $\F$ which cannot be written as $\F=\F' \oplus \F''$ for non-zero coherent sheaves is called \emph{indecomposable}. A coherent sheaf containing no non-zero proper sub-sheaves is called \emph{simple}.

\subsection{Base Change}

Given an $A$--module $M$, and a commutative ring $R$, let
\[
R[M] := \left\{ \sum r_i m_i | m_i \in M, m_i \neq *, r_i \in R \right\}
\]
$R[M]$ naturally inherits the structure of an $R[A]$--module. We may use this to define a base extension functor 
\begin{align*}
QCoh(X) & \rightarrow QCoh(X_R) \\
\mc{F} & \rightarrow \F_R
\end{align*}
It is defined on affines by assigning to $\wt{M}$ on $\mspec{A}$ the quasicoherent sheaf
\[
\wt{M}_R := \wt{R[M]}
\]
on $\mspec(A)_R = \on{Spec}(R[A])$, and for a general monoid scheme by gluing in the obvious way.

 Given a monoid scheme $X$ and $\F \in QCoh(X)$, we have for each open $U \subset X$ a map
\begin{equation} \label{basechange_sec}
\phi_R (U):  R[\Gamma(\F, U)] \rightarrow \Gamma(\F_R, U_R) 
\end{equation}
defined as the unique $R$--linear map with the property that $\phi_R (U) (s) = s \;  \forall s \in \Gamma(\F,U)$. When $U$ is understood, we will refer to this map simply as $\phi_R$.  

\section{Projective schemes over $\fun$} \label{projschemes}

In this section, we briefly recall the construction of the projective monoid scheme $\mproj(A)$ attached to a graded monoid $A$, following \cite{CHWW} (see also \cite{Th1} and \cite{LPL3} for a more general construction in the context of blueprints). It is a straightforward analogue of the $\on{Proj}$ construction for graded commutative rings. 

Let $A = \oplus^{\infty}_{i=0} A_i$ be an $\mathbb{N}$--graded monoid (i.e. $A_i \cdot A_j \subset A_{i+j}$). $A_{\geq 1} = \oplus_{i \geq 1} A_i$ is therefore an ideal, and the map $A \rightarrow A/A_{\geq 1} \simeq A_0$ induces a map $\mspec(A_0) \rightarrow \mspec(A)$, whose image consists of all the prime ideals of $A$ containing $A_{\geq 1}$. Let $\mproj(A)$ denote the topological space  $\mspec (A) \backslash \mspec (A_0)$ with the induced Zariski topology. 

Given a multiplicatively closed subset  $S$ of $A$, the localization $S^{-1} A$ inherits a natural $\mathbb{Z}$--grading by $deg(\frac{a}{b}) = deg(a) - deg(b)$. For an element $f \in A$, let $A_{(f)}$ denote the elements of degree $0$ in the localization $A_f$. Similarly, given a prime ideal $\p \in A$, let $A_{(\p)}$ denote the elements of degree $0$ in $A_{\p}$. For $f \in A$, we may, as in the case or ordinary schemes, identify $D_+ (f) := \mspec(A_{(f)})$ with the open subset $\{ \p \in \mproj(A) \vert f \notin \p \}$ - these cover $\mproj(A)$. Finally, we equip $X = \mproj(A)$ with a monoid structure sheaf $\mc{O}_X$ defined by the property that $\mc{O}_X \vert_{D_+(s)} \simeq \wt{A_{(f)}}$. $(X, \mc{O}_X)$ thus acquires the structure of a monoid scheme locally isomorphic to $(\mspec(A_{(f)}), \mc{O}_{A_{(f)}})$. The stalk of $\mc{O}_{X, \p }$ at $\p \in \mproj(A) $ is $A_{(\p)}$.

\begin{definition}
Let $B$ be a monoid. Let $A=B \langle x_1, x_2, \cdots, x_m \rangle$ denote the graded monoid with $A_0 = B$ and $A_n = \{ b x^{i_1}_1 x^{i_2}_2 \cdots x^{i_m}_m \}$ where $b \in B,$ $i_j \geq 0,$ and $i_1 + i_2 + \cdots + i_m = n $ (i.e. the $x_i$ each have degree 1), with multiplication 
\[
(b x^{i_1}_1 x^{i_2}_2 \cdots x^{i_m}_m) \cdot (b' x^{j_1}_1 x^{j_2}_2 \cdots x^{j_m}_m) = bb' x^{i_1+j_1}_{1} x^{i_2+j_2}_{2} \cdots x^{i_m+j_m}_{m}.
\] 
When $B=\{0,1\}$, we write $B \langle x_1, x_2, \cdots, x_m \rangle$ simply as $\langle x_1, x_2, \cdots, x_m \rangle$.
\end{definition}

\begin{example}
Let $A = \langle t_0, t_1 \rangle$. Then $\mproj(A) \simeq \mathbb{P}^1$. More generally, we define $\mathbb{P}^n$ to be $\mproj( \langle t_0, t_1, \cdots , t_n \rangle)$.
\end{example}

Given a graded monoid $A$ and commutative ring $R$, the monoid algebra $R[A]$ acquires the structure of a graded $R$--algebra (by assigning elements of $R$ degree 0). As shown in \cite{CHWW}, we have

\begin{lemma}
$\mproj(A)_R \simeq \on{Proj}(R[A])$
\end{lemma}

As indicated in Remark \ref{Noeth_rmk}, coherent sheaves are well-behaved only on monoid schemes satisfying a local finiteness condition. We recall this property following \cite{CHWW, CLS}. 

\begin{definition}
A monoid scheme $X$ is of \emph{finite type} if every $x \in X$ has an open affine neighborhood $U_x = \mspec(A_x)$ with $A_x$ a finitely generated monoid. 
\end{definition}

The following lemma is immediate:

\begin{lemma}
\begin{enumerate}
\item Let $A$ be a graded monoid. If $A$ is finitely generated, then $\mproj(A)$ is of finite type.
\item Let $B$ be a finitely generated monoid. Then for every $r \geq 0$, $ \mproj(B \langle x_1, x_2, \cdots, x_m \rangle) $ is of finite type. 
\end{enumerate}
\end{lemma}

\section{Quasicoherent sheaves on projective schemes over $\fun$ } 

Let $A$ be an $\mathbb{N}$--graded monoid, $X = \mproj(A)$, and $M = \oplus_{i \in \mathbb{Z}} M_i$ a $\mathbb{Z}$--graded $A$--module (an $A$--module such that $A_i \cdot M_j \subset M_{i+j}$). For a multiplicatively closed subset $S \subset A$, the localization $S^{-1} M$ inherits a natural grading by $deg(\frac{m}{s}) = deg(m) - deg(s)$, and as in the previous section we use the notation $M_{(f)}$ and $M_{(\p)}$ to denote the degree zero elements in $S^{-1}_{f} M$ and $S^{-1}_{\p} M$ respectively. We may associate to $M$ a quasi-coherent sheaf $\mt$ on $X$ such that for an open subset $U \subset X$, $\mt(U)$ consists of functions $U \rightarrow \sqcup_{\p \in U} M_{(\p)}$ which are locally induced by fractions of the form $\frac{m}{s}$. As for ordinary $\on{Proj}$, one readily checks that $\mt \vert_{ D_+ (f)} \simeq \widetilde{M_{(f)}}$, and that $\mt_{\p} = M_{(\p)}$. If $A$ and $M$ are finitely generated (as a monoid and $A$--module respectively), then $\mt$ is coherent. 

\bigskip
Let $grA-mod$ denote the category of graded $A$--modules whose morphisms are grading-preserving maps of $A$--modules.  The assignment $M \rightarrow \mt$ defines a functor $grA-mod \rightarrow Qcoh(X)$.

%
%\begin{lemma}
%The functor $M \rightarrow \mt$ is strong exact. 
%\end{lemma}
%\begin{proof}

%\end{proof}

Given a $\mathbb{Z}$--graded $A$--module $M$ and $n \in \mathbb{Z}$, let $M(n)$ denote the graded $A$--module defined by $M(n)_i := M_{i+n}$. 

\begin{definition}
For $n \in \mathbb{Z}$, denote by $\mc{O}_X (n)$ the sheaf $\wt{A(n)}$ of $\mc{O}_X$--modules on $X = \mproj(A)$. More generally, for a sheaf $\F$ of $\mc{O}_X$--modules, denote $\F \otimes_{\mc{O}_X} \mc{O}_X (n)$ by $\F(n)$. 
\end{definition}

Let us now assume that $A$ is generated by $A_1$ over $A_0$, and that $A_1$ is finite. $X=\mproj(A)$ then has a finite affine cover of the form $\{ D_+ (f) \}, f \in A_1$, with $\mc{O}_X \vert_{D_+ (f)} \simeq \wt{A_{(f)}}$.

\begin{lemma} Let $X = \mproj(A)$.
\begin{enumerate}
\item The sheaf $\mc{O}_X (n)$ is locally free. 
\item For graded $A$--modules $M$ and $N$, $\wt{M \otimes_A N} \simeq \wt{M} \otimes_{\mc{O}_X} \wt{N}$.  . 
\item  $\wt{M}(n) \simeq \wt{M(n)}$. In particular, $\mc{O}_X (n) \otimes_{\mc{O}_X} \mc{O}_X(m) \simeq \mc{O}_X (m+n)$
\end{enumerate}
\end{lemma}

\begin{proof}
(1) Let $f \in A_1$.  We have $\mc{O}_X (n) \vert_{D_+ (f)} = \wt{A(n)_{(f)}}$, with $A(n)_{(f)} = \{ \frac{a}{f^d} \}$,  where $a \in A(n)_d = A(n+d)$. Since $f$ is invertible in $A_{(f)}$, division by $f^n$ defines an isomorphism of $A_{(f)}$--modules from  $A(n)_{(f)}$ to $A_{(f)}$, inducing an isomorphism of $\mc{O}_X$ modules $$\mc{O}_X (n) \vert_{D_+ (f)} \rightarrow \wt{A_{(f)}} = \mc{O}_X \vert_{D_+ (f)}.$$ Since $A_1$ generates $A$ over $A_0$, the $D_+ (f)$ cover $X$, and so the claim follows.

For (2), we have $\wt{M \otimes_A N} \vert_{D_+ (f)} = \wt{ M \otimes_A N_{(f)}}$, and $\wt{M} \otimes_{\mc{O}_X} \wt{N} \vert_{D_+ (f)} = \wt{M_{(f)} \otimes_{A_{(f)}} N_{(f)} }$. Given $m \in M_i, n \in N_j$, we can send $\frac{m \otimes n}{f^{i+j}} \in M \otimes_A N_{(f)}$ to $\frac{m}{f^i} \otimes \frac{n}{f^j} \in M_{(f)} \otimes_{A_{(f)}} N_{(f)}$. This defines an isomorphism of $A_{(f)}$--modules. These are easily seen to glue, yielding the desired isomorphism of quasicoherent sheaves. 

(3) For the first part, take $N=A(n)$.  The second follows by further specializing to $M = A(m)$. 

\end{proof}

\begin{definition}
Let $A$ be a graded monoid, $X = \mproj(A)$ and $\mc{F}$ a sheaf of $\mc{O}_X$--modules. Let $\Gamma_*  ( \mc{F}) = \oplus_{n \in \mathbb{Z}} \Gamma(X, \mc{F}(n))$. $\Gamma_* (\mc{F})$ has the structure of a $\mathbb{Z}$--graded $A$--module by placing $\Gamma(X, \mc{F}(n))$ in degree $n$, and defining the action $$A_i \times \Gamma(X, \mc{F}(n)) \rightarrow \Gamma(X, \mc{F}(n+i))$$ by identifying $A_i$ with global sections of $\mc{O}_X (i)$ and using the isomorphism $\mc{F}(n) \otimes_{\mc{O}_X} \mc{O}_X (i) \simeq \mc{F}(n+i) $. 
\end{definition}

\begin{theorem}
Let $B$ be a monoid, and $A = B \langle x_0, \cdots, x_r \rangle,\;  r \geq 1$. Let $X = \mproj(A)$. Then $\Gamma(X, \mc{O}_X (n)) = A_n$. 
\end{theorem}

\begin{proof}
We have $\Gamma(D_+ (x_i), \mc{O}_X (n)) = A(n)_{(x_i)}$, which is the set of degree $n$ elements in the localization $A_{x_i}$. Thus a global section of $\mc{O}_X (n)$ consists of a tuple $(t_0, \cdots, t_r), t_i \in A_{x_i}, deg(t_i) = n$, such that $t_i \vert_{D_+ (x_i x_j)} = t_j \vert_{D_+ (x_i x_j)}$ for all $i \neq j$. This is equivalent to requiring that $t_i = t_j$ in $A_{x_i x_j}$. Suppose $g = x^{k_0}_0 \cdots x^{k_r}_r$. Since the $x_i$ are not zero-divisors in $A$, the natural map $A_g \rightarrow A_{g x_i}$ is an inclusion for each $i$. $A_{x_i}$ and $A_{x_i x_j}$ are therefore naturally submonoids of $A_{x_0 \cdots x_r}$. It follows that $\Gamma(X, \mc{O}_X (n)) = \{ h \in \cap^r_{i=0} A_{x_i} \vert deg(g) =n \}$. An element of $A_{x_0 \cdots x_n}$ can be written uniquely in the form $b x^{k_0}_0 \cdots x^{k_r}_r,  b \in B, i_j \in \mathbb{Z}$, and lies in $A_{x_i}$ if and only if $k_j \geq 0$ for all $j \neq i$. The result follows. 

\end{proof}

\begin{corollary}
Let $X = \mproj( B \langle x_0, \cdots, x_r \rangle )$. Then $\Gamma_*(X, \mc{O}_X (n)) = A(n)$. 
\end{corollary}

\begin{proof}
$\Gamma_* (X, \mc{O}_X (n)) = \oplus_{m \in \mathbb{Z}} \Gamma(X, \mc{O}_X (m+n)) = \oplus_{m \in \mathbb{Z}} A_{m+n}$, where $A_{m+n}$ occurs in degree $m$. The result follows. 
\end{proof}

\bigskip

\begin{lemma} \label{affine_extension}
Let $A$ be a monoid, $\F$ a quasicoherent sheaf on $X = \spec (A)$, and $f \in A$. 
\begin{enumerate}
\item If $s_1, s_2 \in \Gamma(X, \F)$ are global sections such that $s_1  \vert_{D(f)} = s_2 \vert_{D(f)}$, then $f^n s_1 = f^n s_2 \in \Gamma(X, \F)$ for some positive integer $n$.
\item If $s \in \Gamma(D(f), \F)$ then $f^n s$ extends to a global section in $\Gamma(X, \F)$ for some positive integer $n$.
\end{enumerate}
\end{lemma}

\begin{proof}
(1) Since $\F$ is quasicoherent, $\F = \wt{M}$ for some $A$--module $M$, and $s_1, s_2$ can be identified with elements of $M$. $s_i \vert_{D(f)} = \frac{s_i}{1} \in M_f$. The hypothesis then implies that $f^n s_1 = f^n s_2 \in M$ by the definition of the localized module $M_f$.  

(2) We can identify $s \in \Gamma(D(f), \F)$ with an element of the form $\frac{m}{f^n} \in M_f$. Then $f^n s = m \in M = \Gamma(X, \F)$. 
\end{proof}

\medskip

\begin{lemma} \label{sections_extend}
Let $A$ be a graded monoid, finitely generated by $A_1$ over $A_0$, $f \in A_1$, and $\F$ a quasicoherent sheaf on $X = \mproj (A)$. 
\begin{enumerate}
\item If $s_1, s_2 \in \Gamma(X, \F)$ are global sections such that $s_1 \vert_{D_+ (f)} = s_2 \vert_{D_+ (f)}$, then there is a positive integer $n$ such that $f^n s_1 = f^n s_2$ viewed as global sections in $\Gamma(X, \F(n))$. 
\item Given a section $s \in \Gamma(D_+ (f), \F)$, there is a positive integer $m$ such that $f^m s$ extends to a global section of $\F(m)$
\end{enumerate}
\end{lemma}

\begin{proof}
(1) Denote by $x_1, \cdots, x_r$ the elements of $A_1$. The sets $D_+ (x_i)$ then form a finite open affine cover of $X$. Let $u_i = \frac{f}{x_i} \in A_{(x_i)}$. The intersection $D_+ (f) \cap D_+ (x_i) = D_+ (f x_i)$ viewed as a subset of $D_+ (x_i) $ then the distinguished open $D(u_i)$.  We have $s_1 = s_2$ on $D(u_i)$, and by part (1) of lemma \ref{affine_extension}, there is a positive integer $n_i$ such that $u^{n_i}_i s_1 = u^{n_i}_i s_2$ on $D_+ (x_i)$. Take $n \geq n_i$ for $i=1, \cdots, r$.

Viewing $\frac{1}{x_i}$ as a section in $\Gamma(D_+(x_i), \mc{O}_X (-1))$, let $\rho_i: \F(n) \vert_{D_+ (x_i)} \rightarrow \F \vert_{D_+ (x_i)}$ be the isomorphism of "dividing by $x^n_i$". I.e. $\rho_{i,n} (s) = s \otimes \frac{1}{x^n_i}$ for $s \in \Gamma(D_+(x_i))$. Viewing $f^n s_i = f^n \otimes s$ as elements in $\Gamma(X, \F(n))$, we have $\rho_{i,n} (f^n s_1) = f^n s_1 \otimes \frac{1}{x^n_i} = u^n_i s_1 = u^n_i s_2 = \rho_{i,n}(f^n s_2)$ for every $i$. It follows that $f^n s_1 = f^n s_2$ in $\Gamma(X, \F(n))$. 

(2) Consider $s \vert_{D_+(x_i) \cap D_+ (f)}$. By part (2) of lemma \ref{affine_extension}, there is a positive integer $n_i$ such that $u^{n_i}_i s = q_i \vert_{D_+(x_i) \cap D_+ (f)}$ for some $q_i \in \Gamma(D_+ (x_i), \F)$. Take $n \geq n_i$ for $i=1, \cdots, r$, and let $t_i = \rho^{-1}_{i,n} (q_i) \in \Gamma(D_+ (x_i), \F(n))$. We have $t_i \vert_{D_+ (f) \cap D_+ (x_i) } = f^n s$, thus $t_i = t_j$ on $D_+(x_i) \cap D_+ (x_j) \cap D_+ (f)$. By part (1), there is a positive integer $k_{ij}$ such that $f^{k_{ij}} t_i = f^{k_{ij}} t_j $ in $\Gamma(D_+(x_i) \cap D_+ (x_j), \F(n+k_{ij}) )$. Taking $k \geq k_{ij}$ for all pairs $i,j$, we see that the sections $f^k t_i \in \Gamma(D_+ (x_i), \F(n+k))$ glue to yield a global section whose restriction to $D_+ (f)$ is $f^{n+k} s$. The result follows taking $m = n +k$.  
\end{proof}

\begin{theorem} \label{mainthm1}
Let $A$ be a graded monoid finitely generated by $A_1$ over $A_0$, and let $X = \mproj(A)$. Given a quasicoherent sheaf $\mc{F}$ on $X$, there exists a natural isomorphism $\beta: \wt{\Gamma_* ( \mc{F})} \simeq \mc{F}$. 
\end{theorem}

\begin{proof}
Let $f_1, \cdots, f_r \in A_1$ be a set of generators for $A$ over $A_0$. Since $\wt{\Gamma_* (\mc{F})}$ is quasicoherent, it suffices to specify isomorphisms of $A_{(f_i)}$--modules $$\beta_i : \Gamma(D_+ (f_i),  \wt{\Gamma_* ( \mc{F})} ) \rightarrow \Gamma(D_+ (f_i), \mc{F}),$$ and check these glue. A section on the left is represented by a fraction of the form $\frac{t}{f_i^d}$, where $t \in \Gamma(X, \mc{F}(d))$. Let $\beta_i(\frac{t}{f_i^d}) = t \otimes f^{-d}$, where $f_i^-d$ is viewed as a section of $\mc{O}_X (-d)$, and we use the isomorphism $\mc{F}(d) \otimes \mc{O}_X (-d) \simeq \mc{F} $. It is immediate that $\beta_i = \beta_j$ on $D_+ (f_i f_j) = D_+ (f_i) \cap D_+ (f_j)$. 

We now verify that $\beta_i$ is an isomorphism for all $i$.  Let $s \in \Gamma(D_+ (f_i), \F)$. By the second part of \ref{sections_extend}, $f^n_i s$ extends to a section  of $\Gamma(X, \F(n))$ for some $n > o$. We have that $\beta_i (f^n_i s) = s$, so $\beta_i$ is surjective. To show injectivity, suppose $\beta_i ( \frac{t_1}{f^d_i}) = \beta_i (\frac{t_2}{f^d_i})$. By the first part of \ref{sections_extend}, there is an $n>0$ such that $f_i^{n-d} t_1 = f_i ^{n-d} t_2 $ as global sections of $\F(n)$. It follows that $\frac{t_1}{f^d_i} = \frac{t_2}{f^d_i}$ in $\Gamma(D_+ (f_i),  \wt{\Gamma_* ( \mc{F})} ) $. 

\end{proof}

\begin{rmk}

We note that just as in the case of ordinary schemes, the graded $A$--module $M$ giving rise to  $\F=\wt{M}$ is not unique. Let  $M_{\geq d} = \oplus_{i \geq d} M_i $. This is a graded $A$--submodule of $M$. Define an equivalence relation $\sim$ on graded $A$--modules by declaring $M \sim M'$ if there exists an integer $d$ such that $M_{\geq d} \simeq M'_{\geq d}$ as graded $A$--modules. We then have the following result, proved exactly as for ordinary schemes. 
\end{rmk}

\begin{lemma}
Let $M, M'$ be two graded $A$--modules such that $M \sim M'$. Then $\wt{M} \simeq \wt{M'}$ as quasicoherent $\mc{O}_X$--modules. 
\end{lemma}

%
%\begin{lemma}
%Let $A$ be a graded monoid, $X=\mproj(A)$, and $\F$ a coherent sheaf. Then there exists a finitely-generated graded $A$--module $M$ such that $\F = \mt$. 
%\end{lemma}
%\begin{proof}

%\end{proof}

%\begin{lemma}
%Let $A_0$ be a finitely generated monoid, and $A$ a graded monoid finitely generated by $A_1$ over $A_0$. Let $M$ be a finitely generated graded $A$--module. Then there exists a filtration
%\[
%0 = M^0 \subset M^1 \subset \cdots \subset M^r=M
%\]
%by graded $A$--modules such that $M^i / M^{i-1} \simeq (A/\mathfrak{p_i})(l_i)$, where $\mathfrak{p_i}$ are prime ideals of $A$ and $l_i \in \mathbb{Z}$. 
%\end{lemma}
%

\medskip
\subsection{Global generation of twists}

\begin{definition} Let $X$ be a monoid scheme and $\F$ and $\mc{O}_X$--module. We say that $\F$ is \emph{generated by} $\{ s_i \}_{i \in I}  \in \Gamma(X,\F)$ if for each $x \in X$, the stalk $\F_x$ is generated by $\{ s_{i,x} \}_{i \in I}$ as an $\mc{O}_{X,x}$--module. 
\end{definition}

\begin{rmk}
If $\F$ is a coherent sheaf on an affine monoid scheme $X = \mspec{A}$ of finite type, then $\F = \wt{M}$ for a finitely generated $A$--module $M = \Gamma(X,\F)$. $\F$ is thus generated by finitely many global sections (the generators of $M$).
\end{rmk}

\medskip

\begin{theorem} \label{fg_sheaf}
Let $A_0$ be a finitely generated monoid, $A$ a graded monoid finitely generated by $A_1$ over $A_0$, and $\F$ a coherent sheaf on $X=\mproj(A)$. Then there exists $n_0$ such that $\F(n)$ is generated by finitely many global sections for all $n \geq n_0$.
\end{theorem}

\begin{proof}
Let $f_1, \cdots, f_r \in A_1$ be a set of generators for $A$ over $A_0$. Since $\F$ is coherent, there is for each $i=1,\cdots, r$ an $m_i \geq 0$ and $s_{i1}, \cdots, s_{im_{i}} \in 
\Gamma((D_+ (f_i), \F)$ which generate $\F \vert_{D_+ (f_i)}$. By Lemma \ref{sections_extend}, there is an $n_0 \geq 0$ such that $f^{n_0} s_{ij}$ extend to global sections in $\Gamma(X,\F(n_0))$ for $1 \leq i \leq r, 0 \leq j \leq m_i$. Since $$\Gamma((D_+(f_i), \F(n)) = f^n_i \otimes \Gamma(D_+(f_i), \F), $$ it follows that $f^n s_{ij}$ generate $\F(n)$ for all $n \geq n_0$.
\end{proof}

\medskip

\begin{corollary} \label{maincor}
With the hypotheses of Theorem \ref{fg_sheaf}, there exist integers $m \in \mathbb{Z}, k \geq 0$, such that $\F$ is a quotient of $\mc{O}_X(m)^{\oplus k}$
\end{corollary}

\begin{proof}
By Theorem \ref{fg_sheaf}, for large enough $n$, $\F(n)$ is generated by global sections, say $k$ of them $s_1, \cdots s_k$. We thus have a surjection of $\mc{O}_X$--modules
\[
\mc{O}^{\oplus k}_X \rightarrow \F(n) \rightarrow 0
\]
Tensoring this sequence with $\mc{O}_X (-n)$, which is locally free, preserves surjectivity on stalks, and we obtain 
\[
\mc{O}_X (-n)^{\oplus k} \rightarrow \F \rightarrow 0,
\]
proving the result. 
\end{proof}

%\subsection{Base Change}

%Let $A$ be a graded monoid and $M$ a graded $A$--module. Given a commutative ring $R$, $R[M]$ inherits the structure of graded $R[A]$--module. We have the following compatibility with base change:

%\begin{lemma}
%$ \mt  $ and $ \wt{R[M]} $ are isomorphic as quasicoherent sheaves on $X_R$
%\end{lemma}

%\begin{proof}
%It suffices to check this on basic open affines. Let $f \in A_n$. Then on $(D_+(f))_R = \on{Spec}(R[A]_{(f)})$ we have that $$\mt_R \vert_{(D_+ (f))_R} \simeq \wt{R[M_{(f)}]} \simeq \wt{ R[M]_{(f)}}  \simeq\wt{R[M]}  \vert_{(D_+ (f))_R}  $$
%\end{proof}
%and one checks readily that these isomorphisms glue. 

\subsection{Finiteness of global sections of coherent sheaves}

One of the key results about coherent sheaves on projective schemes is the finite-dimensionality of the space of global sections. In this section,  we prove the $\fun$--analog of this, in the case that case that $A_0 = \{ 0, 1\}$.

\begin{theorem} \label{mainthm3}
Let $A$ be a graded monoid finitely generated by $A_1$ over $A_0=\langle 0,1 \rangle$, and $\F$ a coherent sheaf on $X=\mproj(A)$.Then $\Gamma(X,\F)$ is a finite pointed set. 
\end{theorem}

\begin{proof}
Let $K$ be a field. Base changing to $\on{Spec}(K)$ yields a coherent sheaf $\F_K$ over the Noetherian projective scheme $X_K$, and so $\Gamma(X_K, F_K) = H^0(X_K,F_K)$ is a finite-dimensional vector space over $K$. Consider the base-change map on global sections
\[
\phi_K:  K[\Gamma(X, \F)] \rightarrow \Gamma(X_K, F_K) 
\]
The codomain of $\phi_K$ is finite-dimensional, and so $Ker(\phi_K)$ must have finite co-dimension in  $\Gamma(X, \F) \otimes K$. Suppose now that $\Gamma(X,\F)$ is an infinite pointed set. Let $f_1, \cdots, f_r \in A_1$ be a set of generators for $A$. $D_+ (f_i)$ therefore form a finite affine cover of $X$. Let $T_i \subset \Gamma(D_+(f_i))$ denote the image of the restriction map $\Gamma(X,\F) \rightarrow \Gamma(D_+ (f_i), \F)$. We have 
\[
\Gamma(X,\F) \subset T_1 \times T_2 \times \cdots \times T_r
\]
and since $\Gamma(X,\F)$ is infinite, at least one of the $T_i$, say $T_1$, is infinite. Let $$T_1=\{ s_1, s_2, \cdots \}.$$ Given $s \in \Gamma(X,\F) \otimes K$, $s \vert D_+ (f_1)$ can be uniquely written as a linear combination of $s_1, s_2, \cdots$. We therefore have maps
\[
C_{s_i}: K[ \Gamma(X,\F)  \rightarrow K
\] 
where $C_{s_i} (s)$ is the coefficient of $s_i$ in $ s \vert D_+ (f_1)$. We have
\[
Ker(\Phi_K) \subset \cap^{\infty}_{i=1} Ker(C_{s_i}) 
\]
Moreover, the chain of subspaces
\[
Ker(C_{s_1}) \supset Ker(C_{s_1}) \cap  Ker(C_{s_2}) \supset  Ker(C_{s_1}) \cap  Ker(C_{s_2}) \cap Ker(C_{s_3}) \supset \cdots
\]
is non-stationary. It follows that $Ker(\phi_K)$ is of infinite codimension, contradicting the fact that $Im(\phi_K)$ is finite-dimensional. 
\end{proof}
\bigskip
\begin{rmk}
 As demonstrated by Example (\ref{firstexample}), $Ker(\phi_K)$ is non-zero in general. 
\end{rmk}

\section{Classification of coherent sheaves on $\mathbb{P}^1$}

In this section, we undertake the classification of coherent sheaves on $\mathbb{P}^1 = \mproj( \langle t_0, t_1 \rangle )$, the simplest projective monoid scheme. A coherent sheaf $\F$ on $\mathbb{P}^1$ is obtained by gluing coherent sheaves on two copies of $\mathbb{A}^1$ along their intersection, so we begin there. 

\begin{rmk}
In \cite{Sz1}, a certain subcategory of \emph{normal sheaves} of $QCoh(\mathbb{P}^1)$ was considered, and used to define the Hall algebra of $\mathbb{P}^1. $
\end{rmk}

\subsection{Coherent sheaves on $\mathbb{A}^1$}

A coherent sheaf $\F$ on $\mathbb{A}^1$ can be described uniquely as $\F = \widetilde{M}$, where $M$ is a finitely generated $\langle t \rangle$--module. We may associate to $M$ a directed graph $\Gamma_M$ whose vertices are the underlying set of $M \backslash *_{M}$, with directed edges from $m$ to $t \cdot m$ for every $m \in M \backslash *_M$. $\Gamma_M$ thus completely describes the isomorphism class of $M$. We note that every vertex of $\Gamma_M$ has at most one out-going edge, and call a vertex a \emph{leaf} if it has no incoming edges, and a \emph{root} if it has no outgoing edges. It follows that elements of $M$ corresponding to leaves of $\Gamma_M$ form a minimal system of generators for $M$ as a $\langle t \rangle$--module. If $M, N$ are $\langle t \rangle$--modules, then $\Gamma_{M\oplus N} = \Gamma_M \amalg \Gamma_N$ - i.e. direct sums of $\langle t \rangle$--modules (or equivalently coherent sheaves on $\mathbb{A}^1$ correspond to disjoint unions of graphs). In view of these observations, the following lemma is obvious:

\begin{lemma}
Let $M$ be a $\langle t \rangle$--module, and $\F = \widetilde{M}$ the corresponding coherent sheaf on $\mathbb{A}^1$. Then
\begin{enumerate}
\item $\F$ is indecomposable iff $\Gamma_M$ is connected.
\item $\Gamma_M$ has finitely many leaves. 
\end{enumerate}
\end{lemma}

The classification of coherent sheaves on $\mathbb{A}^1$ amounts to the classification of the isomorphism classes of the graphs $\Gamma_M$ (up to isomorphism of directed graphs), which was undertaken in \cite{Sz2}. Since every  finitely generated $\langle t \rangle$--module can be uniquely expressed a  finite direct sum of indecomposable ones (up to reordering), it suffices to classify the latter.

\begin{definition} Let $\Gamma$ be a connected directed graph with finitely many leaves, and with each vertex having at most one out-going edge. We say that
\begin{enumerate}
\item $\Gamma$ is of { \emph type 1 } if it is a rooted tree - i.e. the underlying undirected graph of $\Gamma$ is a tree possessing a unique root, such there is a unique directed path from every vertex to the root (see {\bf Figure 1}). 
\item $\Gamma$ is of {\emph type 2} if it is obtained by joining a rooted tree to the initial vertex of $\Gamma_{\langle t \rangle}$ (see {\bf Figure 2}). 
\item $\Gamma$ is of \emph{type 3} if if obtained from a directed cycle by attaching rooted trees ( see {\bf Figure 3}) to an oriented cycle. 
\end{enumerate}
\end{definition}

%\begin{center}
\begin{minipage}{.5\textwidth}
\begin{center}
\begin{tikzpicture}
\draw [ultra thick,->] (0,0) -- (0.9,0.9);
\draw [fill] (0,0) circle [radius=0.1];
\draw [fill] (1,1) circle [radius=0.1];
\draw [ultra thick,->] (2,0) -- (1.1,0.9);
\draw [fill] (2,0) circle [radius=0.1];
\draw [ultra thick,->] (1,1) -- (1.9,1.9);
\draw [fill] (2,2) circle [radius=0.1];
\draw [fill] (4,0) circle [radius=0.1];
\draw [fill] (3,1) circle [radius=0.1];
\draw [ultra thick,->] (4,0) -- (3.1, 0.9);
\draw [ultra thick,->] (3,1) -- (2.1,1.9);
\draw [fill] (2,1) circle [radius=0.1];
\draw [ultra thick,->] (2,1) -- (2,1.9);
\draw [fill] (2,3) circle [radius=0.1];
\draw [ultra thick,->] (2,2) -- (2,2.9);
\draw [fill] (3,2) circle [radius=0.1];
\draw [ultra thick,->] (3,2) -- (2.1,2.9);
\draw [ultra thick,->] (2,3) -- (2,3.9);
\draw [fill] (2,4) circle [radius=0.1];
\node at (2,-1) {Figure 1};
\end{tikzpicture}
\end{center}
\end{minipage}
\begin{minipage}{.5\textwidth}
\begin{center}
\begin{tikzpicture}
\draw [ultra thick,->] (0,0) -- (0.9,0.9);
\draw [fill] (0,0) circle [radius=0.1];
\draw [fill] (1,1) circle [radius=0.1];
\draw [ultra thick,->] (2,0) -- (1.1,0.9);
\draw [fill] (2,0) circle [radius=0.1];
\draw [ultra thick,->] (1,1) -- (1.9,1.9);
\draw [fill] (2,2) circle [radius=0.1];
\draw [fill] (4,0) circle [radius=0.1];
\draw [fill] (3,1) circle [radius=0.1];
\draw [ultra thick,->] (4,0) -- (3.1, 0.9);
\draw [ultra thick,->] (3,1) -- (2.1,1.9);
\draw [ultra thick,->] (2,2) -- (2,2.9);
\draw [fill] (2,3) circle [radius=0.1];
\draw [ultra thick,dotted,->] (2,3) -- (2,4);
\node at (2,-1) {Figure 2};
\end{tikzpicture}
\end{center}
\end{minipage}
%\end{center}

\begin{center}
\begin{tikzpicture}
\draw [ultra thick,->] (0,0) -- (0.9,0);
\draw [ultra thick,->] (1,0) -- (1,0.9);
\draw [ultra thick,->] (1,1) -- (0.1,1);
\draw [ultra thick,->] (0,1) -- (0,0.1);
\draw [ultra thick,->] (0,-1) -- (0,-0.1);
\draw [ultra thick,->] (-1,-1) -- (-0.1,-0.1);
\draw [ultra thick,->] (2,2) -- (1.1,1.1);
\draw [ultra thick,->] (2,3) -- (2,2.1);
\draw [ultra thick,->] (3,2) -- (2.1,2);
\draw [fill] (0,0) circle [radius=0.1];
\draw [fill] (1,1) circle [radius=0.1];
\draw [fill] (1,0) circle [radius=0.1];
\draw [fill] (0,1) circle [radius=0.1];
\draw [fill] (-1,-1) circle [radius=0.1];
\draw [fill] (0,-1) circle [radius=0.1];
\draw [fill] (2,2) circle [radius=0.1];
\draw [fill] (2,3) circle [radius=0.1];
\draw [fill] (3,2) circle [radius=0.1];
\node at (1,-2) {Figure 3};
\end{tikzpicture}
\end{center}

We then have the following classification result:

\begin{theorem}
Let $M$ be a non-trivial finitely generated indecomposable $\langle t \rangle$--module. Then $\Gamma_M$ is of type 1, 2, or 3. 
\end{theorem}

\begin{proof}
When $\Gamma_M$ is a finite graph, It is shown in \cite{Sz2} that it must be of type 1 or 3. We may therefore assume that $\Gamma_M$ is infinite. It is proven in \cite{Sz2} that $\Gamma_M$ contains at most one cycle - necessarily oriented. However, it is clear that if $\Gamma_M$ has a cycle and finitely many leaves, it must be finite.  $\Gamma_M$ is therefore an infinite tree. If $m, m' \in M$ are elements corresponding to leaves of $\Gamma_M$ (and are therefore members of a minimal generating set of $M$), there are $n, n'$, such that $t^n \cdot m = t^{n'} \cdot m'$. Consequently, there is a vertex $v$ of $\Gamma_M$ such that every directed path starting at a leaf eventually passes through $v$. $\Gamma_M$ is then of type 2. 
\end{proof}

Note that a finitely generated $\langle t \rangle$--module $M$ is torsion iff every connected component of $\Gamma_M$ is of type 1, and torsion-free iff every connected component is of type 2 or 3. 

\begin{rmk}
While a type 3 sheaf $\F$ is torsion-free over $\fun$, its base-change $\F_k$ to a field $k$ with sufficiently many roots of unity (in particular, if $k$ is algebraically closed) is a torsion sheaf, supported at $0$ and roots of unity. 
\end{rmk}

\subsection{Coherent sheaves on $\mathbb{P}^1$}

Specifying a coherent sheaf $\F$ on $\mathbb{P}^1$ amounts to specifying coherent sheaves $\F', \F''$ on $$U_1 = \mspec(\langle t \rangle) \simeq \mathbb{A}^1 \textrm{ and } U_2 = \mspec(\langle t^{-1} \rangle) \simeq \mathbb{A}^1$$ respectively, together with a gluing isomorphism $$\varphi: \F' \vert_{U_1 \cap U_2} \simeq \F'' \vert_{U_1 \cap U_2}$$ on $U_1 \cap U_2 = \mspec(\langle t, t^{-1} \rangle)$. We denote the defining tripe  of $\F$ by $(\F', \F'', \phi)$. 

Note that $\langle t, t^{-1} \rangle$ is the infinite cyclic group $\mathbb{Z}$ with a zero element adjoined. Coherent sheaves on $\mspec(\langle t, t^{-1} \rangle)$ therefore correspond to finitely generated $\mathbb{Z}$--sets, and indecomposable coherent sheaves to $\mathbb{Z}$--orbits. We therefore have:

\begin{lemma} The incdecomposable coherent sheaves on $U_1 \cap U_2 =  \mspec(\langle t, t^{-1} \rangle)$ are of the form $\widetilde{N}$, where $N$ is a $\langle t, t^{-1} \rangle$--module of the form $\langle t, t^{-1} \rangle$ or $\langle t, t^{-1} \rangle /  \langle t^k, t^{-k} \rangle$. We denote by $\mc{L}$ and $\mc{C}_k$ the corresponding coherent sheaves.

\end{lemma}

The following result regarding the restriction of coherent sheaves from $\mathbb{A}^1$ to $U_1 \cap U_2$ is immediate:

\begin{lemma}
Suppose $\F=\widetilde{M}$ an indecomposable coherent sheaf on $\mathbb{A}^1$. Then
\begin{enumerate}
\item If $\Gamma_M$ is of type 1, then $\F \vert_{U_1 \cap U_2} \simeq 0$. 
\item If $\Gamma_M$ is of type 2, then $\F \vert_{U_1 \cap U_2} \simeq \mc{L}$. 
\item if $\Gamma_M$ is of type 3 with an oriented cycle of length $k$, then $\F \vert_{U_1 \cap U_2} \simeq \mc{C}_k$. 
\end{enumerate}
\end{lemma}

Note that: 
\begin{itemize}
\item the automorphism group of $\mc{L}$ on $U_1 \cap U_2$ is $\mathbb{Z}$. We denote by $\phi_n: \mc{L} \rightarrow \mc{L}$ the automorphism of $L$ induced by multiplication by $t^n$ on $\langle t, t^{-1} \rangle$. 
\item the automorphism group of $\mc{C}_k$ is $\mathbb{Z}/ k \mathbb{Z}$.. We denote by $\psi_{m}: \mc{C}_k \rightarrow \mc{C}_k $ the automorphism of $\mc{C}_k$ induced by multiplication by $t^m$ on $\langle t, t^{-1} \rangle /  \langle t^k, t^{-k} \rangle$ (note that $\psi_m$ only depends on $m \; (mod \; k)$). 
\end{itemize}

We thus come to our main result in this section. 

\begin{theorem} \label{P1_classification}
Let $\F$ be an indecomposable coherent sheaf on $\mathbb{P}^1$. Then $F$ is described by one of the following defining triples $(\F', \F'', \varphi)$, where $\F'$ and $\F''$ are indecomposable coherent sheaves on $U_1$ and $U_2$ respectively, and
$$\varphi: \F' \vert_{U_1 \cap U_2} \simeq \F'' \vert_{U_1 \cap U_2}$$ is a gluing isomorphism. 
\begin{enumerate}
\item $(\F', 0, 0)$, where $\F'$ is of type 1.
\item $(0', \F'', 0)$, where $\F'$ is of type 1.
\item $(\F', \F'', \phi)$ where $\F'$ and $\F''$ are of type 2. After choosing isomorphisms $\F'_{U_1 \cap U_2} \simeq \mc{L}$, $\F''_{U_1 \cap U_2} \simeq \mc{L}$, $\phi$ may be identified with $\phi_n$ for some $n \in \mathbb{Z}$. 
\item $(\F', \F'', \psi)$ where $\F'$ and $\F''$ are of type 3. After choosing isomorphisms $\F'_{U_1 \cap U_2} \simeq \mc{C}_k$, $\F''_{U_1 \cap U_2} \simeq \mc{C}_k$, $\psi$ may be identified with $\psi_m$ for some $m \in \mathbb{Z}/k \mathbb{Z}$.  
\end{enumerate}
$\F$ is torsion in the first two cases, and torsion-free in the last two. 
\end{theorem}

\begin{example} \label{firstexample}
 Let $M$ be the $\langle t \rangle$--module on two generators with $\Gamma_M$ as shown:
 
 
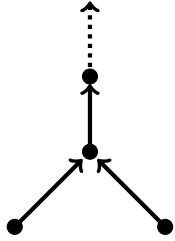
\begin{center}
\begin{tikzpicture}
\draw [ultra thick,->] (0,0) -- (0.9,0.9);
\draw [fill] (0,0) circle [radius=0.1];
\draw [fill] (1,1) circle [radius=0.1];
\draw [ultra thick,->] (2,0) -- (1.1,0.9);
\draw [fill] (2,0) circle [radius=0.1];
\draw [ultra thick,->] (1,1) -- (1,1.9);
\draw  [fill] (1,2) circle [radius=0.1];
\draw [ultra thick,dotted,->] (1,2) -- (1,3);
%\node at (1,-1) {Figure 4};
\end{tikzpicture}
\end{center}
 Let $\F=\wt{M}$ be the corresponding coherent sheaf on $\mathbb{A}^1$, and $\F_1, \F_2$ sheaves on $U_1$ and $U_2$ respectively isomorphic to $\F$. Denote the generators on $U_1$ by $a_0, b_0$ and those on $U_2$ by $a_{\infty}, b_{\infty}$. Consider the coherent sheaf $(\F_1, \F_2, \psi)$, where $\psi: \F_1 \rightarrow \F_2$ identifies the images of $a_0, b_0$ with $a_{\infty}, b_{\infty}$ over $U_1 \cap U_2$. Given a field $K$, we have
 \[
 \F_{K} \cong \mc{O}_{\mathbb{P}^1} \oplus K_{0} \oplus K_{\infty}
 \]
 where $K_0$ and $K_{\infty}$ are torsion sheaves supported at $0$ and $\infty$ isomorphic in local coordinates to $K[ t ] / (t)$ and $K[ t^{-1}]  / (t^{-1})$ respectively. Denoting $s \in \Gamma(\mathbb{P}^1, \F)$ by the pair $(s\vert_{U_1}, s\vert_{U_2})$, the set of global sections consists of $4$ non-zero elements 
 $$(a_0,a_{\infty}), (a_0,b_{\infty}), (b_0,a_{\infty}), (b_0,b_{\infty}), $$
  whereas $H^0 (\mathbb{P}^1_K, \F_K)$ is $3$-dimensional. The base change map 
 \[
 \phi_K : K[ \Gamma (\mathbb{P}^1, \F) ]  \rightarrow \Gamma (\mathbb{P}^1_K, \F_K) 
 \]
 is easily seen to be surjective, with kernel spanned by $$ (a_0,a_{\infty}) - (a_0,b_{\infty}) + (b_0,a_{\infty}) - (b_0,b_{\infty}). $$ 
 
\end{example}

\medskip

\begin{example} \label{secondexample}
Let $P_n$ be the $\langle t \rangle$--module on two generators with $\Gamma_{P_n}$ consisting of an infinite ladder with one additional vertex attached via an incoming edge to the $n$-th vertex from the bottom. I.e. we have $t^n \cdot a = t \cdot b$ as shown:
 
 
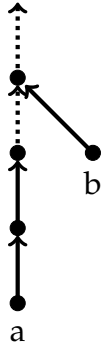
\begin{center}
\begin{tikzpicture}
\draw [ultra thick,->] (0,0) -- (0,0.9);
\draw [fill] (0,0) circle [radius=0.1];
\draw [fill] (0,1) circle [radius=0.1];
\draw [ultra thick,->] (0,1) -- (0,1.9);
\draw [fill] (0,2) circle [radius=0.1];
\draw [ultra thick,dotted,->] (0,2) -- (0,2.9);
\draw [ultra thick,dotted,->] (0,3) -- (0,4);
\draw  [fill] (0,3) circle [radius=0.1];
\draw [ultra thick,->] (1,2) -- (0.1,2.9);
\draw  [fill] (1,2) circle [radius=0.1];
\node at (0,-0.4) {a};
\node at (1,1.6) {b};
\end{tikzpicture}
\end{center}

Let $\F_1 = \wt{P_n}$ on $U_1$ with generators labeled $a_0, b_0$ and $\F_2 = \wt{\langle t^{-1} \rangle}$ on $U_2$, with generator labeled $c_{\infty}$. Consider the coherent sheaf $ \G_n = (\F_1, \F_2, \rho)$, where $\rho$ identifies $a_0$ with $c_{\infty}$ over $U_1 \cap U_2$. The global section $(a_0, c_{\infty})$ generates a sub-module of $\G_n$ isomorphic to $\mc{O}_{\mathbb{P}^1}$, with $\G_n / \mc{O}_{\mathbb{P}^1}$ isomorphic to  the torsion sheaf $\mc{T}=\langle t \rangle / (t)$ (note that the round bracket denotes the \emph{ideal} generated by $t$). We thus obtain infinitely many non-isomorphic non-split extensions
\[
0 \rightarrow \mc{O}_{\mathbb{P}^1} \rightarrow \G_n \rightarrow \mc{T} \rightarrow 0
\]
Upon base-change to a field $K$, $(\G_n)_K \cong \mc{O}_{\mathbb{P}^1} \oplus K_0 $, and all of these short exact sequences become isomorphic to the split extension
\[
0 \rightarrow \mc{O}_{\mathbb{P}^1} \rightarrow  \mc{O}_{\mathbb{P}^1} \oplus K_0 \rightarrow K_0 \rightarrow 0.
\]

\end{example}

\address{\tiny INSTITUTO NACIONAL DE MATEMATICA PURA E APLICADA, ESTRADA DONA CASTORINA 110, RIO
DE JANEIRO, BRAZIL} \\
\indent \footnotesize{\email{oliver@impa.br}}

\address{\tiny DEPARTMENT OF MATHEMATICS AND STATISTICS, BOSTON UNIVERSITY, 111 CUMMINGTON MALL, BOSTON} \\
\indent \footnotesize{\email{szczesny@math.bu.edu}}

\end{document}